\documentclass[twoside, 12pt]{amsart}
\usepackage[letterpaper,hmargin=1.25in,vmargin=1.5in]{geometry}
\usepackage{amscd}
\usepackage{verbatim}
\usepackage{color}
\usepackage{comment}

\input xy
\xyoption{all}

\usepackage{amssymb}
\usepackage{hyperref}

\DeclareMathAlphabet{\cat}{OT1}{cmss}{m}{sl}

\newtheorem*{theorem*}{Theorem}
\newtheorem{theorem}{Theorem}[section]

\newtheorem{proposition}[theorem]{Proposition}
\newtheorem{lemma}[theorem]{Lemma}
\newtheorem{corollary}[theorem]{Corollary}

\theoremstyle{definition}
\newtheorem{remark}[theorem]{Remark}

\newcommand{\xra}{\xrightarrow}

\newcommand{\tens}{\otimes}

\newcommand{\gmu}{\boldsymbol{\mu}}

\newcommand{\CH}{\operatorname{CH}}

\renewcommand{\Im}{\operatorname{Im}}

\newcommand{\Ker}{\operatorname{Ker}}

\newcommand{\Pic}{\operatorname{Pic}}

\newcommand{\ind}{\operatorname{\hspace{0.3mm}ind}}
\newcommand{\ch}{\operatorname{char}}
\newcommand{\Inv}{\operatorname{Inv}}
\newcommand{\res}{\operatorname{res}}

\newcommand{\Br}{\operatorname{Br}}
\newcommand{\Spec}{\operatorname{Spec}}

\newcommand{\SB}{\operatorname{SB}}

\newcommand{\gPGL}{\operatorname{\mathbf{PGL}}}
\newcommand{\gSL}{\operatorname{\mathbf{SL}}}

\newcommand{\gSO}{\operatorname{\mathbf{SO}}}
\newcommand{\gGL}{\operatorname{\mathbf{GL}}}
\newcommand{\gm}{\operatorname{\mathbb{G}}_m}

\newcommand{\tors}{\operatorname{-\cat{torsors}}}
\newcommand{\rank}{\operatorname{rank}}

\newcommand{\Sdec}{\operatorname{Sdec}}
\newcommand{\Dec}{\operatorname{Dec}}
\newcommand{\torsion}{\operatorname{tors}}
\newcommand{\norm}{\operatorname{norm}}

\renewcommand{\P}{\mathbb{P}}
\newcommand{\Z}{\mathbb{Z}}

\newcommand{\Q}{\mathbb{Q}}

\title[Chow groups of products of Severi-Brauer varieties and invariants] 
{Chow groups of products of Severi-Brauer varieties and invariants of degree $3$}

\author
[S. Baek] {Sanghoon Baek}

\address{Department of Mathematical Sciences, KAIST, 291 Daehak-ro, Yuseong-gu, Daejeon, 305-701, Republic of Korea}

\email {sanghoonbaek@kaist.ac.kr}

\begin{document}

\maketitle

\begin{abstract}
We study the semi-decomposable invariants of a split semisimple group and their extension to a split reductive group by using the torsion in the codimension $2$ Chow groups of a product of Severi-Brauer varieties. In particular, for any $n\geq 2$ we completely determine the degree $3$ invariants of a split semisimple group, the quotient of $(\gSL_{2})^{n}$ by its maximal central subgroup, as well as of the corresponding split reductive group. We also provide an example illustrating that a modification of our method can be applied to find the semi-decomposable invariants of a split semisimple group of type A.
\end{abstract}


\section{Introduction}\label{intro}

Let $G$ be a linear algebraic group over a field $F$. The notion of a cohomological invariant of $G$ was introduced by Serre \cite{GMS} as follows. Consider the functor $G\tors$ taking a field $K/F$ to the set $G\tors(K)$ of isomorphism classes of $G$-torsors over $\Spec(K)$. A degree $d$ (cohomological) \emph{invariant} of $G$ with values in a cohomology group $H^{d}(K):=H^{d}(K, C)$ for some Galois module $C$ is a morphism of functors $G\tors\to H^{d}$. All the invariants of degree $d$ of $G$ form a group $\Inv^{d}(G)$. We refer to \cite{GMS} and \cite{Skip} for various examples and general discussion of invariants.

In this paper, we work with the Galois module $C=\Q/\Z(d-1)$, i.e., the Galois cohomology functor $H^{d}(-):=H^{d}(-,\Q/\Z(d-1))$, where $\Q/\Z(d-1):=\bigoplus_{p} \varinjlim_{n} \gmu_{p^{n}}^{\tens d-1}$ if $\ch(F)\neq p$ and $\bigoplus_{p} \varinjlim_{n} W_{n}\Omega^{d-1}_{\log}[1-d]$ otherwise ($W_{n}\Omega^{d-1}_{\log}$ is the logarithmic de Rham-Witt sheaf). In particular, one has $H^{2}(K)=\Br(K)$.

From now on we focus on degree $3$ invariants. The group $\Inv^{3}(G)$ was determined by Rost in the case where $G$ is simply connected quasi-simple \cite{GMS}. In \cite{MerDeg} Mekurjev constructed an exact sequence which generalize Rost's result to an arbitrary semisimple group. In particular, when $G$ is an adjoint group of inner type the group $\Inv^{3}(G)$ was computed by means of \emph{decomposable} and \emph{indecomposable} invariants.

Recently, a new type of degree $3$ invariant of a split semisimple group, called a \emph{semi-decomposable} invariant was introduced by Merkurjev-Neshitov-Zainoulline \cite{MNZ}. This invariant is locally decomposable, so any decomposable invariant is semi-decomposable. In their paper, authors established a relation between nontrivial semi-decomposable invariants and the torsion in the codimension $2$ Chow groups of a generic flag variety and showed that there is no nontrivial semi-decomposable invariant of a split simple group in \cite[Theorem]{MNZ}.

The first and, so far, the only example of nontrivial semi-decomposable invariants comes from a semisimple group $\gSO_{4}\simeq (\gSL_{2}\times \gSL_{2})/\gmu$, where $\gmu=\{(\lambda_{1}, \lambda_{2})\in \gmu_{2}\times \gmu_{2}\,|\, \lambda_{1}\lambda_{2}=1\}$; see \cite[Example 3.1]{MNZ}. Indeed, this invariant is given by $\phi:=a\langle\langle b, c\rangle\rangle\mapsto (a)\cup [b, c]$, where $\phi$ is a $4$-dimensional quadratic form with trivial discriminant over a field extension $K/F$ and $[b, c]$ is the class of a quaternion algebra in the Brauer group $\Br(K)$; see \cite[Example 20.3]{GMS}. Since the class $[b, c]$ is the image of $\phi$ under the Hasse-Witt invariant, the invariant is semi-decomposable. On the other hand, it is not decomposable as it is nontrivial over an algebraic closure of $F$.

In the present paper, we completely determine the degree $3$ invariants of split semisimple groups $(\gSL_{2})^{n}/\gmu$ for all $n\geq 2$ (see Theorem below), which generalize the above example (observe that there is no indecomposable degree $3$ invariant of $\gSO_{n}$ for $n\geq 5$ \cite[\S 19, 20]{GMS}). Furthermore, we also determine the degree $3$ invariants of split reductive groups $(\gGL_{2})^{n}/\gmu$ for any $n\geq 2$ and provide an explicit description of all semi-decomposable invariants. Namely, our main result reads as follows.

\begin{theorem*}\label{mainthm}
Let $\gmu=\{(\lambda_{1}, \ldots, \lambda_{n})\in (\gmu_{2})^{n}\,|\, \lambda_{1}\cdots \lambda_{n}=1\}$ be a central subgroup of $(\gGL_{2})^{n}$ over a field $F$ and let $G=(\gGL_{2})^{n}/\gmu$ and $H=(\gSL_{2})^{n}/\gmu$ be factor groups. Then, for any $n\geq 2$ the group $\Inv^{3}(H)_{\ind}$ of indecomposable invariants of $H$ is cyclic of order $2$ containing the subgroups $\Inv^{3}(G)_{\ind}=\Z/2\Z$ if $n\geq 3$ and $0$ otherwise,
\[\frac{\Sdec(H)}{\Dec(H)}=
\begin{cases} \Z/2\Z  & \text{ if } 2\leq n\leq 4,\\
			0       & \text{ otherwise, }
\end{cases}
\text{ and }\quad
\frac{\Sdec(G)}{\Dec(G)}=
\begin{cases} \Z/2\Z  & \text{ if } 3\leq n\leq 4,\\
			0       & \text{ otherwise, }
\end{cases}
\]
i.e., the semi-decomposable invariants of $H$ $($respectively, $G$$)$ coincide with the decomposable invariants if $n\geq 5$ $($respectively, $n=2$ or $n\geq 5$$)$.
Moreover, if $\operatorname{char}(F)\neq 2$, then the semi-decomposable invariant of $G$ $($but not decomposable$)$ is given by
\[\begin{cases}(Q_{1}, Q_{2}, Q_{3})\mapsto [Q_{1}]\cup (a)+[Q_{2}]\cup (-1) & \text{ if }n=3,\\
(Q_{1}, Q_{2}, Q_{3}, Q_{4})\mapsto [Q_{1}]\cup (b)+[Q_{2}]\cup (-b)+[Q_{3}]\cup (-1) & \text{ if }n=4,
\end{cases}
\]
where $Q_{1},\ldots, Q_{4}$ are quaternions over a field extension $K/F$ and $a, b\in K^{\times}$.
\end{theorem*}

In the proof of the Theorem we compute the torsion in codimension $2$ cycles of a product of Severi-Brauer varieties associated to quaternion algebras and relate it to the torsion of the generic variety of $(\gSL_{2})^{n}/\gmu$. We then combine it with the group of indecomposable invariants of $(\gSL_{2})^{n}/\gmu$ and $(\gGL_{2})^{n}/\gmu$ using the short exact sequence in \cite[Theorem]{MNZ} and its extension to reductive groups (Proposition \ref{firstpropsection2}).

We would like to emphasize that our method developed in the proof of the theorem can be applied to find nontrivial semi-decomposable invariants of an arbitrary split semisimple group of type A: In the last section, using the same methods developed so far, we  present an example (Proposition \ref{lastprop}) of a semisimple group which has a nontrivial semi-decomposable invariant.


The paper is organized as follows. In Section \ref{semidecompandinvariant}, we extend the notion of semi-decomposable invariants of split semisimple groups to the class of split reductive groups and provide its connection to the invariants of semisimple groups. In Sections \ref{compdegreeGL} and \ref{sectionfour}, we determine the indecomposable invariants of our groups in Theorem and find all their semi-decomposable invariants. In Section \ref{sectionfivech}, we compute the torsion in $\CH^{2}$ on the product of Severi-Brauer varieties associated to the generic variety. In Section \ref{sectionforpf}, we combine all the results from the previous sections and prove our main result. In the last section, we discuss applications of our method.

\section{Semi-decomposable invariants of a split reductive group}\label{semidecompandinvariant}

Let $G$ be a linear algebraic group over a field $F$. An invariant $\alpha\in \Inv^{d}(G)$ is called \emph{normalized} if $\alpha(Z)=0$ for the trivial $G$-torsor $Z$, and we write $\Inv^{d}(G)_{\norm}$ for the subgroup of normalized invariants in $\Inv^{d}(G)$. Then, we have an isomorphism $\Inv^{d}(G)\simeq \Inv^{d}(G)_{\norm}\oplus H^{d}(F)$.

If $d=1$, then $\Inv^{1}(G)_{\operatorname{norm}}$ is trivial for a connected group $G$. For $d=2$, it was shown in \cite[Theorem 2.4]{BM} that the group $\Inv^{2}(G)_{\norm}$ is isomorphic to the Picard group $\Pic(G)$ for a smooth connected group $G$. In particular, for a split reductive group $G$ and its derived subgroup $H$ one has
\begin{equation}\label{invariantdegreetwo}
\Inv^{2}(G)_{\norm}\simeq \Inv^{2}(H)_{\norm}\simeq M^{*}, \text{ thus }\Inv^{2}(G)\simeq \Inv^{2}(H),
\end{equation}
where $M^{*}$ is the character group of the kernel of the universal covering morphism of $H$ (see \cite[Example 2.1]{MerUnramified}). Indeed, the isomorphism $M^{*}\stackrel{\sim}{\to} \Inv^{2}(H)_{\norm}$ is given by the composition of the connecting morphism $\delta$ with the induced morphism $\chi_{*}: H^{2}(K, M)\to H^{2}(K, \gm)=\Br(K)$:
\begin{equation}\label{degreetwonormalized}
H\tors(K)\stackrel{\delta}{\to}H^{2}(K, M)\stackrel{\chi_{*}}{\to} \Br(K)
\end{equation}
for each character $\chi\in M^{*}(K)$ over a field extension $K/F$.

For $d=3$, the group $\Inv^{3}(G)_{\norm}$ is finite cyclic with canonical generator (called the Rost invariant) if $G$ is simply connected quasi-simple \cite[Theorem 9.11]{GMS}. For an arbitrary semisimple group $G$, the group $\Inv^{3}(G)_{\norm}$ is determined by means of an exact sequence of five terms \cite[Theorem]{MerDeg}. In particular, it was calculated for all adjoint groups of inner type \cite{MerDeg} and for the remaining split simple groups \cite{BR}.

Let $G$ be a split reductive group over $F$. A normalized invariant $\alpha\in \Inv^{3}(G)_{\norm}$ is called \emph{decomposable} if it is given by a cup product of an invariant in degree $2$ with a constant invariant in degree $1$ \cite[\S 3]{MerDeg}. The subgroup of decomposable invariants will be denoted by $\Dec(G)$. Observe that $\alpha\in \Dec(G)$ if and only if $\alpha_{\bar{F}}$ is trivial. We shall write $\Inv^{3}(G)_{\ind}=\Inv^{3}(G)_{\norm}/\Dec(G)$ and call $\Inv^{3}(G)_{\ind}$ the \emph{indecomposable} invariants \cite[\S 3]{MerDeg}.

Let $H$ be a split semisimple group over $F$. In \cite{MNZ} a \emph{semi-decomposable} invariant of $H$ was introduced. We write $\Sdec(H)$ for the subgroup of semi-decomposable invaraints. This invariant can be viewed as a locally decomposable invariant so that $\Dec(H)\subseteq\Sdec(H)\subseteq \Inv^{3}(H)_{\norm}$.

Consider a generically free representation $V$ of $H$, i.e., there exists a nonempty $H$-invariant open subset $U\subseteq V$ such that $U\to U/H$ is a $H$-torsor. Then, the generic fiber $\bar{U}\to \Spec F(U/H)$ of $U\to U/H$ is a \emph{versal} $H$-torsor \cite{GMS}. Let $B$ be a Borel subgroup of $H$ and let $\bar{X}$ be the associated flag variety $\bar{U}/B$ over $\Spec F(U/H)$, which is called the \emph{generic variety} of $H$ \cite{MNZ}. Then, it was shown in \cite[Theorem]{MNZ} that there is an exact sequence
\begin{equation}\label{mainMNZ}
0\to\frac{\Sdec(H)}{\Dec(H)}\to\Inv^{3}(H)_{\ind}\to\CH^{2}(\bar{X})_{\torsion}\to 0.
\end{equation}
In particular, it was proved that $\Dec(H)=\Sdec(H)$ if $H$ is simple.

We extend the definition of semi-decomposable invariants of split semisimple groups \cite[Definition 2.9]{MNZ} to the class of split reductive groups: let $G$ be a split reductive group. A normalized invariant $\alpha$ of degree $3$ of $G$ is \emph{semi-decomposable} if for any field extension $K/F$ and any $Z\in G\tors(K)$ there exist $\beta_{i}\in \Inv^{2}(G)_{\norm}$ and $x_{i}\in K^{\times}$ such that
\[\alpha(Z)=\sum_{i}\beta_{i}(Z)\cup (x_{i}).\]

We extend the exact sequence (\ref{mainMNZ}) to the class of split reductive groups:
\begin{proposition}\label{firstpropsection2}
Let $G$ be a split reductive group over a field $F$ and let $H$ be its $($semisimple$)$ derived subgroup. Then,  there is a commutative diagram
\begin{equation*}
\xymatrix{
0 \ar@{->}[r] & \frac{\Sdec(H)}{\Dec(H)} \ar@{->}[r] & \Inv^{3}(H)_{\ind} \ar@{->}[r]  & \CH^{2}(\bar{X})_{\torsion}  \ar@{=} [d]\ar@{->}[r] & 0\\
0 \ar@{->}[r] & \frac{\Sdec(G)}{\Dec(G)} \ar@{->}[r]  \ar@{^{(}->} [u]& \Inv^{3}(G)_{\ind}\ar@{->}[r] \ar@{^{(}->} [u]& \CH^{2}(\bar{X})_{\torsion} \\
}
\end{equation*}
with the exact rows, where $\bar{X}$ is the generic variety associated to $H$.
\end{proposition}
\begin{proof}
Consider the exact sequence $1\to H\xra{\phi} G\to T\to 1$, where $T$ is a split torus induced by the embedding $\phi$. Then, for a field extension $K/F$ this yields a surjective morphism
\begin{equation}\label{surjGH}
\phi_{*}: H\tors(K)\to G\tors(K) \text{ given by } Z\mapsto \phi_{*}(Z):=(Z\times G)/H,
\end{equation}
which yields an injective morphism
\[\bar{\phi}: \Inv^{d}(G)\to \Inv^{d}(H) \text{ given by } \bar{\phi}(\alpha)(Z):=\alpha(\phi_{*}(Z))\]
for any degree $d$. Hence, we have the injectivity of the middle column (see \cite[Proposition 6.1]{MerUnramified}).

It remains to show that $\Sdec(G)=\Sdec(H)\cap \Inv^{3}(G)$. Let $\alpha\in \Sdec(G)$. Then, for $Z\in H\tors(K)$ there exist $\beta_{i}\in \Inv^{2}(G)_{\norm}$ and $x_{i}\in K^{\times}$ such that
\[\bar{\phi}(\alpha)(Z)=\alpha(\phi_{*}(Z))=\sum_{i} \beta_{i}(\phi_{*}(Z))\cup (x_{i})=\sum_{i}\bar{\phi}(\beta_{i})(Z)\cup (x_{i}),\]
thus $\bar{\phi}(\alpha)\in \Sdec(H)$. Conversely, let $\alpha\in \Inv^{3}(G)$ such that $\bar{\phi}(\alpha)\in \Sdec(H)$. It follows from (\ref{surjGH}) that for $Y\in G\tors(K)$ there exists $Z\in H\tors(K)$ such that $\phi_{*}(Z)=Y$. Therefore, we have
\begin{equation}\label{trivialproof}
\alpha(Y)=\bar{\phi}(\alpha)(Z)=\sum_{i} \beta_{i}(Z)\cup (x_{i})
\end{equation}
for some $\beta_{i}\in \Inv^{2}(G)_{\norm}$ and $x_{i}\in K^{\times}$. Since $\bar{\phi}$ is an isomorphism by (\ref{invariantdegreetwo}), there exists $\gamma_{i}\in \Inv^{2}(G)$ such that $\bar{\phi}(\gamma_{i})=\beta_{i}$, thus the rightmost term in (\ref{trivialproof}) equals to $\sum_{i} \gamma_{i}(Y)\cup (x_{i})$. Hence, $\alpha\in \Sdec(G)$.\end{proof}

\begin{remark}
In general, the map $\Inv^{3}(G)_{\ind}\to \CH^{2}(\bar{X})_{\torsion}$ in the bottom row of Proposition \ref{firstpropsection2} is not necessarily surjective. Consder the reductive group $G=\gGL_{8}/\gmu_{2}$ and its derived group $H=\gSL_{8}/\gmu_{2}$. It follows from \cite[Theorem 2.8]{BM09} that $\Inv^{3}(G)_{\ind}=0$. On the other hand, by \cite[Example 4.15]{Kar} we have $\CH^{2}(X)_{\torsion}=\Z/2\Z$. Note that by  \cite[Theorem 4.1]{BR} we also have $\Inv^{3}(H)_{\ind}=\Z/2\Z$.
\end{remark}

\section{Degree $3$ invariants of $(\gGL_{2})^{n}/\gmu$}\label{compdegreeGL}

In the present section, we give an explicit description of semi-decomposable invariants of $(\gGL_{2})^{n}/\gmu$ for $n=3, 4$, which in turn are all nontrivial semi-decomposable invariants of $(\gGL_{2})^{n}/\gmu$ for any $n\geq 2$.

For $n\geq 2$, consider a central subgroup $\gmu=\{(\lambda_{1}, \ldots, \lambda_{n})\in (\gmu_{2})^{n}\,|\, \lambda_{1}\cdots \lambda_{n}=1\}$ of $(\gGL_{2})^{n}$. An element of $(\gGL_{2})^{n}/\gmu\tors(K)$, where $K/F$ is a field extension, is the class of $n-$tuples $(Q_{1},\ldots, Q_{n})$ of quaternions over $K$ satisfying the relation $[Q_{1}]+\cdots+[Q_{n}]=0$ in $\Br(K)$.

The reduced norm form $N_{Q}=\langle\langle a, b\rangle\rangle$ of a quaternion algebra $Q=(a, b)$ over $K$ is in $I^{2}(K)$ and the image of $N_{Q}$ under $e_{2}:I^{2}(K)\to \Br_{2}(K)$ is the class $[Q]$. Therefore, the sum of norm forms $\sum_{i=1}^{n} N_{Q_{i}}$ of an element $(Q_{1},\ldots, Q_{n})$ of $(\gGL_{2})^{n}/\gmu\tors(K)$ is contained in $I^{3}(K)=\Ker(e_{2})$.

Let $\alpha_{n}$ be a degree $3$ invariant of $(\gGL_{2})^{n}/\gmu$ defined by
\begin{equation}\label{alphainvariant}
(Q_{1},\ldots, Q_{n})\mapsto e_{3}(\sum_{i=1}^{n} N_{Q_{i}}),
\end{equation}
where $e_{3}:I^{3}(K)\to H^{3}(K)$ is the Arason invariant for a field extension $K/F$ (see \cite[Example 11.2]{MerUnramified}). If $n=2$, then any element $(Q_{1}, Q_{2})\in (\gGL_{2})^{2}/\gmu\tors(K)$ satisfies $Q_{1}\simeq Q_{2}$, thus the invariant $\alpha_{2}(Q_{1}, Q_{2})=[Q_{1}]\cup (-1)$ in $H^{3}(K)$ is decomposable as it is trivial over an algebraic closure $\bar{K}$ of $K$.

\begin{lemma}\label{Gthree}
Let $(Q_{1}, Q_{2}, Q_{3})\in (\gGL_{2})^{3}/\gmu\tors(K)$. Then, the invariant $\alpha_{3}$, which can be written explicitly as $(Q_{1}, Q_{2}, Q_{3})\mapsto [Q_{1}]\cup (q)+[Q_{2}]\cup (-1)$ for some $q\in K^{\times}$, is semi-decomposable but not decomposable.
\end{lemma}
\begin{proof}
First, observe that in the Witt ring $W(K)$ one has
\begin{equation}\label{twofold}
\langle\langle x, y\rangle\rangle + \langle\langle x, z\rangle\rangle=\langle\langle x, y, z\rangle\rangle+\langle\langle x, yz\rangle\rangle
\end{equation}
for $x, y, z\in K^{\times}$.

Since $Q_{1}\tens Q_{2}$ is not a division algebra, they have a common splitting quadratic extension. Therefore, $Q_{1}=(a, b)$ and $Q_{2}=(a, c)$ for some $a, b, c\in K^{\times}$. It follows from the relation $[Q_{1}]+[Q_{2}]+[Q_{3}]=0$ that $Q_{3}=(a, bc)$.

Applying (\ref{twofold}) to the sum of norm forms, one obtain
\[\sum_{i=1}^{3} N_{Q_{i}}=\langle\langle a, b\rangle\rangle+ \langle\langle a, c\rangle\rangle+ \langle\langle a, bc\rangle\rangle=\langle\langle a, b, c\rangle\rangle+ \langle\langle a, bc\rangle\rangle+ \langle\langle a, bc\rangle\rangle\]
in $W(K)$. As $2\cdot \langle\langle a, bc\rangle\rangle=\langle\langle a, bc, -1\rangle\rangle$, we have
\[\alpha_{3}(Q_{1}, Q_{2}, Q_{3})=[Q_{1}]\cup (-c)+[Q_{2}]\cup (-1)\quad \text{ in }H^{3}(K),\]
which is not decomposable as it is nontrivial over $\bar{K}$. Obviously, this invariant is semi-decomposable.
\end{proof}

\begin{remark}
Lemma \ref{Gthree} shows that the invariant $\alpha_{n}$ is nontrivial for $n\geq 3$.
\end{remark}

\begin{proposition}\label{invariantalphafour}
Let $\operatorname{char}(F)\neq 2$. The invariant $\alpha_{4}$ of $(\gGL_{2})^{4}/\gmu$, which can be written explicitly as $(Q_{1}, Q_{2}, Q_{3}, Q_{4})\mapsto [Q_{1}]\cup (q)+[Q_{2}]\cup (-q)+[Q_{3}]\cup (-1) $ for some $q\in K^{\times}$, is semi-decomposable but not decomposable.
\end{proposition}
\begin{proof}
Let $Q_{1}=(a, b)$, $Q_{2}=(c, d)$, $Q_{3}=(e, f)$, $Q_{4}=(g, h)$ be quaternions over $K$ such that $[Q_{1}]+[Q_{2}]=[Q_{3}]+[Q_{4}]$ in $\Br(K)$. Then, by the chain lemma for biquaternion algebras \cite[Proposition 1]{Siv} one can find $x, y, z\in K$ such that
\begin{equation}\label{simpleaceg}
[ac, x]=[eg, y]=[ac, z]=0
\end{equation}
and
\[[Q_{3}']=[Q_{1}']+[a, z],\quad [Q_{4}']=[Q_{2}']+[c, z]\]
in $\Br(K)$, where $Q_{1}'=(a, bx),\, Q_{2}'=(c, dx),\, Q_{3}'=(e, fy),\, Q_{4}'=(g, hy)$ are quaternion algebras over $K$. Hence,
\begin{equation}\label{qoneqtwoqoneprime}
[Q_{1}]+[Q_{2}]=[Q_{1}']+[Q_{2}'] \text{ and } [Q_{3}]+[Q_{4}]=[Q_{3}']+[Q_{4}'] \text{ in } \Br(K).
\end{equation}

Applying (\ref{twofold}), we have
\begin{align}
N_{Q_{1}}+N_{Q_{2}}+N_{Q_{1}'}+N_{Q_{2}'}&=  \langle\langle a, b\rangle\rangle+\langle\langle c, d\rangle\rangle+\langle\langle a, bx\rangle\rangle+\langle\langle c, dx\rangle\rangle  \nonumber\\
&=  \langle\langle a, x\rangle\rangle+\langle\langle a, b, bx\rangle\rangle+\langle\langle c, x\rangle\rangle+\langle\langle c, d, dx\rangle\rangle \nonumber\\
&= \langle\langle x, ac\rangle\rangle+\langle\langle a, c, x\rangle\rangle+\langle\langle a, b, bx\rangle\rangle+\langle\langle c, d, dx\rangle\rangle\label{sumoffour}
\end{align}
in $W(K)$. By using the relations (\ref{simpleaceg}),
\[\langle\langle a, b, bx\rangle\rangle=\langle\langle a, b, x\rangle\rangle+\langle\langle a, b, -1\rangle\rangle \text{, and } \langle\langle c, d, dx\rangle\rangle=\langle\langle c, d, -x\rangle\rangle,\]
we see that the sum (\ref{sumoffour}) equals
\begin{equation}\label{simpleacegggg}
\langle\langle a, b, x\rangle\rangle+\langle\langle c, d, -x\rangle\rangle+\langle\langle a, c, x\rangle\rangle+\langle\langle a, b, -1\rangle\rangle
\end{equation}
in $W(K)$. By (\ref{simpleaceg}) and $\langle\langle a, a\rangle\rangle=\langle\langle a, -1\rangle\rangle$, one has $\langle\langle a, c, x\rangle\rangle=\langle\langle a, x, -1\rangle\rangle$ thus by applying (\ref{twofold}) one obtain
\begin{equation}\label{simpleacegg}
\langle\langle a, c, x\rangle\rangle+\langle\langle a, b, -1\rangle\rangle=\langle\langle a, bx, -1\rangle\rangle+\langle\langle a, b, x, -1\rangle\rangle.
\end{equation}
in $W(K)$. It follows from (\ref{simpleacegggg}) and (\ref{simpleacegg}) that
\begin{equation}\label{Qonetwo}
\alpha_{4}(Q_{1}, Q_{2}, Q_{1}', Q_{2}')=[Q_{1}]\cup (x)+[Q_{2}]\cup (-x)+[Q_{1}']\cup (-1)
\end{equation}
in $H^{3}(K)$ as a $4$-fold Pfister form vanishes under $e_{3}$.

Similarly, one has
\begin{equation}\label{Qonetwot}
\alpha_{4}(Q_{1}', Q_{2}', Q_{3}', Q_{4}')=[Q_{1}']\cup (z)+[Q_{2}']\cup (-z)+[Q_{3}']\cup (-1)
\end{equation}
and
\begin{equation}\label{Qonetwott}
\alpha_{4}(Q_{3}, Q_{4}, Q_{3}', Q_{4}')=[Q_{3}]\cup (y)+[Q_{4}]\cup (-y)+[Q_{3}']\cup (-1)
\end{equation}
in $H^{3}(K)$.

Combining (\ref{Qonetwo}) with (\ref{Qonetwot}) and (\ref{Qonetwott}) and using the relation (\ref{qoneqtwoqoneprime}) we obtain
\begin{align}
\alpha_{4}(Q_{1}, Q_{2}, Q_{3}, Q_{4})&=\alpha_{4}(Q_{1}, Q_{2}, Q_{1}', Q_{2}')+\alpha_{4}(Q_{1}', Q_{2}', Q_{3}', Q_{4}')+\alpha_{4}(Q_{3}, Q_{4}, Q_{3}', Q_{4}')\nonumber\\
&=[Q_{1}]\cup (x)+ [Q_{2}]\cup (-x)+([Q_{1}]+[Q_{2}])\cup (-z)+[Q_{3}']\cup (1)\nonumber\\
&+[Q_{3}]\cup (y)+[Q_{4}]\cup(-y) \label{sumoffourfff}
\end{align}
in $H^{3}(K)$. Since $[Q_{3}']\cup (1)=0$ in $H^{3}(K)$ and $[Q_{4}]=[Q_{1}]+[Q_{2}]+[Q_{3}]$ in $\Br(K)$, it follows from (\ref{sumoffourfff}) that
\[\alpha_{4}(Q_{1}, Q_{2}, Q_{3}, Q_{4})=[Q_{1}]\cup (xyz)+ [Q_{2}]\cup (-xyz)+[Q_{3}]\cup (-1) \text{ in } H^{3}(K),\]
which is obviously semi-decomposable. This invariant is not decomposable as it is nontrivial over $\bar{K}$.
\end{proof}

\begin{corollary}
Assume that $F$ is an algebraically closed field. Then, the invariant $\alpha_{4}$ vanishes over a field extension $K/F$ if and only if the element $q\in K^{\times}$ is contained in the group of similarity factors of an Albert form associated to $Q_{1}\tens Q_{2}$.
\end{corollary}
\begin{proof}
It follows from Proposition \ref{invariantalphafour} that $\alpha_{4}(Q_{1}, Q_{2}, Q_{3}, Q_{4})=([Q_{1}]+[Q_{2}])\cup (q)$. Let $\phi=\langle a, b, -ab, -c, -d, cd\rangle$ be the Albert form corresponding to $Q_{1}\tens Q_{2}$ for $Q_{1}=(a, b)$ and $Q_{2}=(c, d)$. By \cite[Theorem]{KLST}, $\alpha_{4}$ vanishes if and only if $\langle q\rangle\cdot \phi\simeq \phi$.
\end{proof}

\section{Degree $3$ invariants of $(\gSL_{2})^{n}/\gmu$}\label{sectionfour}

In the present section, we determine the indecomposable invariants of $(\gSL_{2})^{n}/\gmu$ for all $n\geq 2$. For $n=2$, we have $(\gSL_{2}\times \gSL_{4})/\gmu\simeq \gSO_{4}$, thus it follows by \cite[Example 20.3]{GMS} that $\Inv^{3}((\gSL_{2}\times\gSL_{2})/\gmu)_{\ind}=\Z/2\Z$.

\begin{proposition}\label{propsl}
For any $n\geq 2$, we have $\Inv^{3}((\gSL_{2})^{n}/\gmu)_{\ind}=\Z/2\Z$.
\end{proposition}
\begin{proof}
Let $G=(\gGL_{2})^{n}/\gmu$ and $\gmu=\{(\lambda_{1}, \ldots, \lambda_{n})\in (\gmu_{2})^{n}\,|\, \lambda_{1}\cdots \lambda_{n}=1\}\subset G$. Consider the commutative diagram of exact sequences
\begin{equation}\label{CGdiagram}
\xymatrix{
1 \ar@{->}[r] & \gmu \ar@{=}[d]\ar@{->}[r] & (\gmu_{2})^{n} \ar@{_{(}->} [d]\ar@{->}[r]  & \gmu_{2}  \ar@{_{(}->} [d]\ar@{->}[r] & 1\\
1 \ar@{->}[r] & \gmu\ar@{->}[r]  & \bar{T}_{G}\ar@{->}[r]& T_{G} \ar@{->}[r] & 1\\
}
\end{equation}
where $\bar{T}_{G}$ (respectively, $T_{G}$) is a split maximal torus of $(\gGL_{2})^{n}$ (respectively, $G$). This yields the commutative diagram of the corresponding character groups
\begin{equation}\label{TGdiagram}
\xymatrix{
0 \ar@{->}[r] & \Z/2\Z \ar@{->}[r] & (\Z/2\Z)^{n} \ar@{->}[r]  & \gmu^{*}  \ar@{=}[d]\ar@{->}[r] & 0\\
0 \ar@{->}[r] & T_{G}^{*} \ar@{->>}[u]\ar@{->}[r]  & \Z^{2n}\ar@{->>}[u]\ar@{->}[r]& \gmu^{*} \ar@{->}[r] & 0\\
}
\end{equation}
where the map $\Z^{2n}\to (\Z/2\Z)^{n}$ is given by $(a_{1}, b_{1}, \ldots, a_{n}, b_{n})\mapsto (\overline{a_{1}+b_{1}}, \ldots, \overline{a_{n}+b_{n}})$.
It follows from (\ref{TGdiagram}) that $T_{G}^{*}=\{\sum a_{i}x_{i}+b_{i}y_{i} \,|\, \overline{a_{i}+b_{i}}=\overline{a_{j}+b_{j}} \text{ for all } 1\leq i\neq j\leq n\}$, where $\{x_{i}, y_{i}\}$ is a standard basis of $\Z^{2n}$. Hence, we have the following $\Z$-basis of $T_{G}^{*}$
\begin{equation}\label{TGbasis}
\{x_{i}-y_{i},\, 2x_{k},\, \sum_{i}x_{i}\, |\, 1\leq i\leq n,\, 1\leq k\leq n-1\}.
\end{equation}

Let $H=(\gSL_{2})^{n}/\gmu$. We write $\bar{T}_{H}$ (respectively, $T_{H}$) for a split maximal torus of $(\gSL_{2})^{n}$ (respectively, $H$). Then, we have the following commutative diagram of the character groups
\begin{equation}\label{TGHdiagram}
\xymatrix{
0 \ar@{->}[r] & T_{H}^{*} \ar@{->}[r] & \Z^{n} \ar@{->}[r]  & \gmu^{*}  \ar@{=}[d]\ar@{->}[r] & 0\\
0 \ar@{->}[r] & T_{G}^{*} \ar@{->>}[u]\ar@{->}[r]  & \Z^{2n}\ar@{->>}[u]\ar@{->}[r]& \gmu^{*} \ar@{->}[r] & 0\\
}
\end{equation}
where the middle map $\Z^{2n}\to (\Z)^{n}$ is defined by $x_{i}\mapsto \bar{x}_{i}$ and $y_{i}\mapsto -\bar{x}_{i}$. Hence, by (\ref{TGbasis}) we obtain the following $\Z$-basis of $T_{H}^{*}$
\begin{equation}\label{THbasis}
\{2\bar{x}_{k},\, \sum_{i}\bar{x}_{i}\, |\, 1\leq i\leq n,\, 1\leq k\leq n-1\}.
\end{equation}

Let $W=(S_{2})^{n}$ be the Weyl group of $H$, which acts on $\bar{x}_{i}$ by $\bar{x}_{i}\mapsto -\bar{x}_{i}$. Let $\phi$ be a quadratic form on (\ref{THbasis}) over $\Z$, i.e., $\phi\in S^{2}(T_{H}^{*})^{W}$. Since the group $S^{2}(\bar{T}_{H}^{*})^{W}$ is generated by $\{-\bar{x}_{1}^{2}, \ldots, -\bar{x}_{n}^{2}\}$, there exist $d_{i}\in \Z$ for $1\leq i\leq n$ such that
\begin{equation*}
\phi(z_{1},\ldots, z_{n})=\sum_{k}d_{k}\big(\frac{z_{k}}{2}\big)^{2}+d_{n}\big(z_{n}-\sum_{k}\frac{z_{k}}{2}\big)^{2},
\end{equation*}
where $z_{k}=2\bar{x}_{k}$ for $1\leq k\leq n-1$ and $z_{n}=\sum_{i}\bar{x}_{i}$. Thus, we have
\begin{equation}\label{Stwogen1}
\begin{cases} d_{1}+d_{2}\equiv 0\mod 4  & \text{ if } n=2,\\
		d_{i}\equiv 0\mod 2\text{ and } d_{i}\equiv d_{j}\mod 4       & \text{ if } n\geq 3
\end{cases}
\end{equation}
for all $1\leq i\neq j\leq n$. Hence, if $n\geq 3$, then the group $S^{2}(T_{H}^{*})^{W}$ is generated by
\begin{equation}\label{Stwogen2}
\{4\bar{x}_{k}^{2}, \, 2\sum_{i}\bar{x}_{i}^{2}\}.
\end{equation}

Let $c_{2}:\Z[T_{H}^{*}]^{W}\to S^{2}(T_{H}^{*})^{W}$ be the Chern class map defined in \cite[3c]{MerDeg}. Then, it follows from (\ref{THbasis}) that the image of $c_{2}$ is generated by $c_{2}(e^{2\bar{x}_{i}}+e^{-2\bar{x}_{i}})$ for $1\leq i\leq n$ and $c_{2}(\prod_{i} (e^{\bar{x}_{i}}+e^{-\bar{x}_{i}}))$, i.e.,
\begin{equation}\label{Stwodgen2}
\{4\bar{x}_{i}^{2},\, 2^{n-1}\sum_{i}\bar{x}_{i}^{2}\}.
\end{equation}
Therefore, it follows by \cite[Theorem]{MerDeg}, (\ref{Stwogen1}), (\ref{Stwogen2}), and (\ref{Stwodgen2}) that the indecomposable group of $H$ is isomorphic to $\Z/2\Z$ for any $n\geq 2$.\end{proof}


\section{Torsion in $\CH^{2}$ of generic variety associated to $(\gSL_{2})^{4}/\gmu$}\label{sectionfivech}

In this section, we determine the torsion in codimension $2$ cycles of the generic variety of $(\gSL_{2})^{n}/\gmu$ for $n=3, 4$ by using the filtrations on the Grothendieck ring. It turns out that the result for $n=4$ suffices to determine the torsion for all $n\geq 4$ (see Remark \ref{SLFOURD}).

Let $X$ be a smooth projective homogeneous variety over a field $F$ and consider the Grothendieck ring $K_{0}(X)$ of $X$. We set $\Gamma^{0}(X)=K_{0}(X)$ and $\Gamma^{1}(X)=\Ker(\rank:K_{0}(X)\to \Z)$. The gamma filtration
\[K_{0}(X)=\Gamma^{0}(X)\supset \Gamma^{1}(X)\supset \cdots \]
is given by the ideals $\Gamma^{d}(X)$ generated by the product $\gamma_{d_{1}}(x_{1})\cdots \gamma_{d_{i}}(x_{i})$ with $x_{i}\in \Gamma^{1}(X)$ and $d_{1}+\cdots +d_{i}\geq d$, where $\gamma_{d_{j}}$ is the gamma operation on $K_{0}(X)$. For example, we have $\gamma_{0}(x)=1$ and $\gamma_{1}(x)=x$, where $x\in K_{0}(X)$. Indeed, the gamma operation defines the Chern class $c_{i}(x):=\gamma_{i}(x-\rank(x))$ with values in $K_{0}$.

For any $d\geq 0$, we write $\Gamma^{d/d+1}(X)$ for the subsequent quotient $\Gamma^{d}(X)/\Gamma^{d+1}(X)$. Let $E$ be a splitting field of $X$. By a diagram chasing in the diagram involving $0\to \Gamma^{d+1}(X)\to \Gamma^{d}(X)\to \Gamma^{d/d+1}(X)\to 0$ and the one over $E$, one has the following formula \cite[Proposition 2]{Kar95}:
\begin{equation}\label{alphalem}
|\!\oplus \Gamma^{d/d+1}(X)_{\torsion}|\cdot |K_{0}(X_{E})/K_{0}(X)|=\prod_{d=1}^{\dim(X)}|\Gamma^{d/d+1}(X_{E})/\Im(\res^{d/d+1})|,
\end{equation}
where $\res^{d/d+1}: \Gamma^{d/d+1}(X)\to \Gamma^{d/d+1}(X_{E})$ is the restriction map.

We shall write $\operatorname{T}^{d}(X)$ (respectively, $\operatorname{T}^{d/d+1}(X)$) for the topological filtration of degree $d$ given by the ideal generated by the structure sheaf of a closed subvariety of $X$ of codimension at least $d$ (respectively, the quotient $\operatorname{T}^{d}(X)/\operatorname{T}^{d+1}(X)$). For $1\leq d\leq 2$, we have $\Gamma^{d}(X)=\operatorname{T}^{d}(X)$. Moreover, $\Gamma^{2/3}(X)\twoheadrightarrow \operatorname{T}^{2/3}(X)=\CH^{2}(X)$ as $\Gamma^{3}(X)\subseteq \operatorname{T}^{3}(X)$.

Now we consider the product of Severi-Brauer varieties $X=\prod_{i=1}^{n}\SB(A_{i})$ of a central simple $F$-algebra $A_{i}$ of degree $d_{i}$. Let $x_{i}$ be the pullback of the class of the tautological line bundle on $\P^{d_{i}-1}_{E}$ over a splitting field $E$. Then, by  \cite[\S 8 Theorem 4.1]{Qui} the image of the restriction homomorphism $K_{0}(X)\to K_{0}(\prod_{i=1}^{n}\P^{d_{i}-1}_{E})\simeq \Z[x_{1},\cdots, x_{n}]/((x_{1}-1)^{d_{1}},\cdots, (x_{n}-1)^{d_{n}})$ is generated by
\begin{equation}\label{Quillenbasis}
\{ \ind(A_{1}^{\tens i_{1}}\tens \cdots \tens A_{n}^{\tens i_{n}})\cdot x_{1}^{i_{1}}\cdots x_{n}^{i_{n}} \,|\,\, 0\leq i_{j}\leq d_{j}-1, 1\leq j\leq n\}.
\end{equation}
Let $X'=\prod_{i=1}^{n'} \SB(A'_{i})$ be another product of Severi-Brauer varieties of simple $F$-algebras $A_{i}'$ such that $\langle A'_{1}, \ldots, A'_{n'}\rangle=\langle A_{1},\ldots, A_{n}\rangle$ in the Brauer group $\Br(F)$. Then, by applying projective bundle theorem one has (see \cite[Proposition 4.7]{IzhKar})
\begin{equation}\label{reductionSB}
\CH^{2}(X)_{\torsion}\simeq \CH^{2}(X')_{\torsion}.
\end{equation}

Let $H$ be a split semisimple group over $F$ and let $B$ be a Borel subgroup of $H$. We denote by $\bar{X}$ be the generic variety of $H$ as in Section \ref{semidecompandinvariant}. By the localization sequence, the induced pullback of $\bar{X}\hookrightarrow U/B$ gives
\[\CH(U/T)\simeq \CH(U/B)\twoheadrightarrow \CH(\bar{X}), \]
where $T$ is a split maximal torus of $H$ containing $B$. Since the characteristic map $S(T^{*})\to \CH(U/T)$ is an isomorphism, where $S(T^{*})$ denotes the symmetric algebra of the character group $T^{*}$, the Chow group $\CH(U/T)$ is generated by Chern classes of line bundles, thus $\CH(\bar{X})$ is also generated by Chern classes. Therefore, it follows from the commutativity of the diagram \cite[Lemma 2.16]{Kar}
$$
\xymatrix{
K_{0}(\bar{X})\,\,\,\, \ar[rr]^{c_{d}}\ar[d]^{\gamma_{d}\circ (\operatorname{id}-\operatorname{rank})} & & \CH^{d}(\bar{X})\ar@{>>}[d] \\
\Gamma^{d/d+1}(\bar{X}) \ar[rr] & & \operatorname{T}^{d/d+1}(\bar{X})
}
$$
that $\operatorname{T}^{d/d+1}(\bar{X})_{\torsion}=\Gamma^{d/d+1}(\bar{X})_{\torsion}$. In particular, one has
\begin{equation}\label{CHtwoTop}
\CH^{2}(\bar{X})_{\torsion}=\operatorname{T}^{2/3}(\bar{X})_{\torsion}=\Gamma^{2/3}(\bar{X})_{\torsion}.
\end{equation}

\begin{lemma}$($cf. \cite[Proposition 3.1]{Baek}$)$\label{fourgeneric}
Let $\bar{X}_{4}$ be the generic variety associated to $(\gSL_{2})^{4}/\gmu$. Then, the torsion subgroup $\CH^{2}(\bar{X}_{4})_{\torsion}$ is trivial.
\end{lemma}
\begin{proof}
The generic variety $\bar{X}_{4}$ is the product of conics $\prod_{i=1}^{4}\SB(Q_{i})$ over the function field $K$ of a classifying space of $(\gSL_{2})^{4}/\gmu$, where $Q_{1}, Q_{2}, Q_{3}, Q_{4}$ are quaternion algebras over $K$ satisfying
\[\ind(Q_{1}\tens Q_{2}\tens Q_{3}\tens Q_{4})=1, \quad \ind(Q_{i}\tens Q_{j}\tens Q_{k})=2, \quad \ind(Q_{i}\tens Q_{j})=4, \quad \ind(Q_{i})=2 \]
for all $1\leq i\neq j\neq k\leq 4$.

Let $X=\prod_{i=1}^{3}\SB(Q_{i}')$, where $Q_{1}', Q_{2}', Q_{3}'$ are quaternion $K$-algebras such that
\[\ind(Q_{1}'\tens Q_{2}'\tens Q_{3}')=2, \quad \ind(Q_{i}'\tens Q_{j}')=4, \quad \ind(Q_{i}')=2 \]
for all $1\leq i\neq j \leq 3$. By (\ref{reductionSB}), it suffices to show $\CH^{2}(X)_{\torsion}=0$.

By (\ref{Quillenbasis}), we have a basis $\{1, 2x_{i}, 4x_{i}x_{j}, 2x_{1}x_{2}x_{3}\}$ of $K_{0}(X)$, where $x_{i}$ is the pullback of the tautological line bundle on the projective line over a splitting field $E$ of $X$. Hence, $|K_{0}(X_{E})/K_{0}(X)|=2^{10}$. By substitution $y_{i}=x_{i}-1$ for all $1\leq i\leq 3$ we have another basis of $K_{0}(X)$
\begin{equation}\label{threequaternionbasis}
\{1, 2y_{i}, 4y_{i}y_{j}, 2(y_{1}y_{2}y_{3}+y_{1}y_{2}+y_{1}y_{3}+y_{2}y_{3})\}.
\end{equation}

Let $\epsilon_{n}=|\Gamma^{n/n+1}(X_{E})/\Im(\res^{n/n+1})|$ for $1\leq n\leq 3$. Then, we obtain $\epsilon_{1}\leq 2^{3}$. Since $c_{2}(2x_{1}x_{2}x_{3})=6y_{1}y_{2}y_{3}+2(y_{1}y_{2}+y_{1}y_{3}+y_{2}y_{3})\in \Gamma^{2}(X)$ and $c_{2}(2x_{1}x_{2}x_{3})\cdot 2y_{1}=4y_{1}y_{2}y_{3}\in \Gamma^{3}(X)$, it follows from (\ref{threequaternionbasis}) that $\epsilon_{2}\leq 4^{2}\cdot 2$ and $\epsilon_{3}\leq 4$, respectively. Therefore, $|\!\oplus \Gamma^{n/n+1}(X)_{\torsion}|\leq 1$, i.e., $\Gamma^{2/3}(X)_{\torsion}=0$, thus $\CH^{2}(X)_{\torsion}=0$.
\end{proof}

\begin{proposition}\label{fivegeneric}
Let $\bar{X}_{5}$ be the generic variety associated to $(\gSL_{2})^{5}/\gmu$. Then, we have $\CH^{2}(\bar{X}_{5})_{\torsion}=\Z/2\Z$.
\end{proposition}
\begin{proof}
Let $X=\prod_{i=1}^{4}\SB(Q_{i}')$, where $Q_{1}', Q_{2}', Q_{3}', Q_{4}'$ are quaternion algebras over a field $K/F$ such that
\[\ind(Q_{1}'\tens Q_{2}'\tens Q_{3}'\tens Q_{4}')=2, \quad \ind(Q_{i}'\tens Q_{j}'\tens Q_{k}')=4, \quad \ind(Q_{i}'\tens Q_{j}')=4 \]
for all $1\leq i\neq j\neq k\leq 4$. We first show that $\Gamma^{2/3}(X)_{\torsion}=\Z/2\Z$.

Let $x_{i}$ be the pullback of the tautological line bundle on $\P^{1}$ over a splitting field $E$ of $X$. By (\ref{Quillenbasis}), we have a basis $\{1, 2x_{i}, 4x_{i}x_{j}, 2x_{i}x_{j}x_{k}, 2x_{1}x_{2}x_{3}x_{4}\}$ of $K_{0}(X)$. Therefore, $|K_{0}(X_{E})/K_{0}(X)|=2^{4}\cdot 4^{6}\cdot 4^{4}\cdot 2$. Consider another basis
\[\{1, 2y_{i}, 4y_{i}y_{j}, 4y_{i}y_{j}y_{k}, 2(y_{1}y_{2}y_{3}y_{4}+\sum y_{i}y_{j}y_{k}+\sum y_{i}y_{j})\}\]
of $K_{0}(X)$, where $y_{i}=x_{i}-1$. Let $\epsilon_{n}=|\Gamma^{n/n+1}(X_{E})/\Im(\res^{n/n+1})|$ for $1\leq n\leq 4$. It is obvious that $\epsilon_{1}\leq 2^{4}$. As $c_{2}(2x_{1}x_{2}x_{3}x_{4})=14y_{1}y_{2}y_{3}y_{4}+6\sum y_{i}y_{j}y_{k}+2\sum y_{i}y_{j}\in \Gamma^{2}(X)$ and $8y_{1}y_{2}y_{3}y_{4}=2y_{1}\cdot 2y_{2}\cdot c_{2}(2x_{1}x_{2}x_{3}x_{4})\in \Gamma^{4}(X)$, we have $\epsilon_{2}\leq 4^{5}\cdot 2$ and $\epsilon_{4}\leq 2^{3}$, respectively. It follows from the computation of $2y_{i}\cdot c_{2}(2x_{1}x_{2}x_{3}x_{4})$ that
\[z_{l}:=4(y_{1}y_{2}y_{3}y_{4}+\sum_{1\leq p<q<r\leq 4} y_{p}y_{q}y_{r})-4y_{i}y_{j}y_{k}\in \Gamma^{3}(X)\]
for all $l$ such that $\{i, j, k, l\}=\{1, 2, 3, 4\}$. Hence, $4(y_{1}y_{2}y_{3}y_{4}+y_{i}y_{j}y_{k})\in \Gamma^{3}(X)$, thus $\epsilon_{3}\leq 4^{3}$. We conclude from (\ref{alphalem}) that $|\!\oplus \Gamma^{n/n+1}(X)_{\torsion}|\leq 2$.

As $\Gamma^{3}(X)$ is generated by $\Gamma^{1}(X)\cdot \Gamma^{2}(X)$, $c_{2}(4x_{i}x_{j})=12y_{i}y_{j}$, and $c_{2}(4x_{i}x_{j}x_{k})=36y_{i}y_{j}y_{k}+12(y_{i}y_{j}+y_{i}y_{k}+y_{j}y_{k})$, one can easily see that $\Gamma^{3}(X)$ is generated by
\[\{8y_{1}y_{2}y_{3}y_{4}, 8y_{i}y_{j}y_{k}, z_{l}\}\]
for all $i, j, k, l$ satisfying $\{i, j, k, l\}=\{1, 2, 3, 4\}$. Hence, if $4y_{1}y_{2}y_{3}y_{4}\in \Gamma^{3}(X)$, then there exist $f_{m}$ in the image of $K_{0}(X)$ (as an element of $\Z[y_{1}, y_{2}, y_{3}, y_{4}]/(y_{1}^{2}, y_{2}^{2}, y_{3}^{2}, y_{4}^{2})$) for $1\leq m\leq 9$ such that
\begin{equation}\label{generatedbyy}
4y_{1}y_{2}y_{3}y_{4}=(\sum_{\{i, j, k, l\}=\{1, 2, 3, 4\}}f_{l}\cdot 8y_{i}y_{j}y_{k})+f_{5}z_{1}+f_{6}z_{2}+f_{7}z_{3}+f_{8}z_{4}+f_{9}\cdot 8y_{1}y_{2}y_{3}y_{4}
\end{equation}
(indeed, we may assume that $f_{i}\in \Z$). By comparing coefficients of (\ref{generatedbyy}), we obtain
\[8(f_{1}+f_{2}+f_{3}+f_{4}+f_{9})+16(f_{5}+f_{6}+f_{7}+f_{8})=4,\]
which is impossible. Therefore, $4y_{1}y_{2}y_{3}y_{4}\not\in \Gamma^{3}(X)$.

Since $4y_{1}y_{2}y_{3}y_{4}\not\in \Gamma^{3}(X)$ and $4(y_{1}y_{2}y_{3}y_{4}+y_{i}y_{j}y_{k})\in \Gamma^{3}(X)$, all the elements $4y_{i}y_{j}y_{k}$ are not contained in $\Gamma^{3}(X)$. On the other hand, it follows from the computation of $2c_{2}(4x_{1}x_{2}x_{3}x_{4})-3z_{l}$ that $4y_{i}y_{j}y_{k}\in \Gamma^{2}(X)$. Therefore, the class of $4y_{i}y_{j}y_{k}$ generates a subgroup of $\Gamma^{2/3}(X)_{\torsion}$ of order $2$ as $8y_{i}y_{j}y_{k}\in \Gamma^{3}(X)$.

As the generic variety $\bar{X}_{5}$ is $\prod_{i=1}^{5}\SB(Q_{i})$, where $Q_{i}$ is a quaternion algebra over the function field $K$ of a classifying space of $(\gSL_{2})^{4}/\gmu$ such that
\[\ind(\prod_{i=1}^{5}Q_{i})=1,\, \ind(Q_{i}\tens Q_{j}\tens Q_{k}\tens Q_{l})=2,\, \ind(Q_{i}\tens Q_{j}\tens Q_{k})=\ind(Q_{i}\tens Q_{k})=4\]
for all $1\leq i\neq j\neq k\neq l\leq 5$, the result follows from $\Gamma^{2/3}(X)_{\torsion}=\Z/2\Z$, (\ref{reductionSB}), and (\ref{CHtwoTop}).\end{proof}

\begin{remark}
Note that it was shown that $|\!\oplus \Gamma^{n/n+1}(X)_{\torsion}|=2$ in \cite[Theorem 3.5]{Baek}.
\end{remark}


\section{Proof of theorem}\label{sectionforpf}

It follows from Proposition \ref{propsl} and Lemma \ref{Gthree} that $\Inv^{3}(H)_{\ind}=\Z/2\Z$ for $n\geq 2$ and $\Inv^{3}(G)_{\ind}=\Z/2\Z$ for $n\geq 3$. To verify the remaining parts of Theorem we apply Propositions \ref{invariantalphafour}, \ref{propsl}, \ref{fivegeneric} and Lemmas \ref{Gthree}, \ref{fourgeneric} to Proposition \ref{firstpropsection2}. The proof is divided into four cases.

{\it Case}: $n=2$. Since $G\tors(K)\simeq \gPGL_{2}\tors(K)$ for any field extension $K/F$, it follows from \cite[\S 4b]{MerDeg} that $\Inv^{3}(G)_{\ind}=0$, thus $\Sdec(G)/\Dec(G)=0$. On the other hand, it is easy to see that the torsion subgroup $\CH^{2}(\bar{X}_{2})_{\torsion}$ of the corresponding generic variety $\bar{X}_{2}$ is trivial. Hence, it follows from Propositions \ref{firstpropsection2} and \ref{propsl} that $\Sdec(H)/\Dec(H)=\Z/2\Z$.

{\it Case}: $n=3$ or $4$. By (\ref{reductionSB}), the torsion subgroup $\CH^{2}(\bar{X}_{3})_{\torsion}$ of the generic variety $\bar{X}_{3}$ is isomorphic to the torsion subgroup of a flag variety of dimension $2$, thus $\CH^{2}(\bar{X}_{3})=0$. Therefore, it follows from Lemma \ref{fourgeneric} and Proposition \ref{propsl} that $\Sdec(H)/\Dec(H)=\Z/2Z$. On the other hand, the second statement of Theorem follows from Lemma \ref{Gthree} and Proposition \ref{invariantalphafour}, thus $\Sdec(G)/\Dec(G)=\Z/2\Z$.

{\it Case}: $n=5$. It follows from Propositions \ref{propsl} and \ref{fivegeneric} that $\Sdec(H)=\Dec(H)$. Therefore, by Proposition \ref{firstpropsection2} we also have $\Sdec(G)=\Dec(G)$.

{\it Case}: $n\geq 6$. Consider the nontrivial invariant $\alpha_{n}$ in (\ref{alphainvariant}). We claim that if $\alpha_{n}$ is semi-decomposable but not decomposable, then so is $\alpha_{n-1}$. Consider the restriction map $r: \Inv^{3}((\gGL_{2})^{n}/\gmu)\to \Inv^{3}((\gGL_{2})^{n-1}/\gmu)$ defined by
\[r(\alpha)(Q_{1},\ldots, Q_{n-1})=\alpha(Q_{1},\ldots, Q_{n-1}, 0)\]
for $(Q_{1},\ldots, Q_{n-1})\in (\gGL_{2})^{n-1}/\gmu\tors(K)$ and $\alpha\in \Inv^{3}((\gGL_{2})^{n}/\gmu)$. Then, we have $r(\alpha_{n})(Q_{1},\ldots, Q_{n-1})=e_{3}(\sum_{i=1}^{n-1} N_{Q_{i}})=\alpha_{n-1}(Q_{1},\ldots, Q_{n-1})$, i.e., $r(\alpha_{n})=\alpha_{n-1}$. Assume that $\alpha_{n}$ is semi-decomposable. Then, it follows from (\ref{degreetwonormalized}) and (\ref{TGdiagram}) that there exist $\chi_{i}=(b_{1i},\ldots, b_{ni})\in (\Z/2\Z)^{n}/\Z/2\Z=\gmu^{*}$ and $x_{i}\in K^{\times}$ such that
\[\alpha_{n}(Q_{1},\ldots, Q_{n})=\sum_{i}(b_{1i}[Q_{1}]+\cdots+b_{ni}[Q_{n}])\cup (x_{i}).\]
Hence, $\alpha_{n-1}=r(\alpha_{n})$ is semi-decomposable. To show that $\Sdec(G)=\Dec(G)$ we use induction on $n$ together with claim and the previous case. Finally, it follows from Proposition \ref{firstpropsection2} that $\Sdec(H)=\Dec(H)$.

\begin{remark}\label{SLFOURD}
The proof of Theorem shows that for any $n\geq 5$ the torsion subgroup $\CH^{2}(\bar{X}_{n})_{\torsion}$ of the generic variety associated to $\gSL_{n}/\gmu$ is isomorphic to $\Z/2\Z$.
\end{remark}

\section{Additional example of a nontrivial semi-decomposable invariant}

In the present section, we give an example illustrating that a minor modification of our method can be used to determine the semi-decomposable invariants.

\begin{lemma}\label{lemmaforslfour}
Let $\gmu=\{(\lambda_{1}, \lambda_{2})\in (\gmu_{4})^{2}\,|\, \lambda_{1}^{2}\lambda_{2}^{2}=1\}$ be a central subgroup of $\gSL_{4}\times \gSL_{4}$. Then, $\Inv^{3}((\gSL_{4}\times \gSL_{4})/\gmu)_{\ind}=\Z/2\Z$.
\end{lemma}

\begin{proof}
Basically, we follow the proof of Proposition \ref{propsl}. Let $H=(\gSL_{4}\times \gSL_{4})/\gmu$ and $G=(\gGL_{4}\times \gGL_{4})/\gmu$. The same commutative diagram as (\ref{CGdiagram}) with $(\gmu_{2})^{n}$ replaced by $(\gmu_{4})^{2}$ yields the same commutative diagram of the corresponding character groups as (\ref{TGdiagram}) with $(\Z/2\Z)^{n}$ and $\Z^{2n}$ replaced by $\Z/4\Z\oplus \Z/4\Z$ and $\Z^{4}\oplus \Z^{4}$, where the map $\Z/4\Z\oplus \Z/4\Z\to \Z^{4}\oplus \Z^{4}$ is given by
\[(a_{1}, a_{2}, a_{3}, a_{4}, b_{1}, b_{2}, b_{3}, b_{4})\mapsto (\overline{a_{1}+a_{2}+a_{3}+a_{4}}, \overline{b_{1}+b_{2}+b_{3}+b_{4}}).\]
Hence, $T_{G}^{*}=\{\sum a_{i}x_{i}+b_{i}y_{i} \,|\, \overline{\sum a_{i}}=\overline{\sum b_{i}}=\bar{0} \text{ or } \overline{\sum a_{i}}=\overline{\sum b_{i}}=\bar{2}\text{ for all } 1\leq i\leq 4\}$, where $\{x_{i}, y_{i}\}$ is a standard basis of $\Z^{4}\oplus \Z^{4}$. Therefore, we have the following $\Z$-basis of $T_{G}^{*}$:
\[\{x_{1}-x_{2}, x_{1}-x_{3}, x_{1}-x_{4}, y_{1}-y_{2}, y_{1}-y_{3}, y_{1}-y_{4}, 2x_{1}+2y_{1}, 2x_{1}-2y_{1}\}.\]

Consider the same diagram of the character groups as (\ref{TGHdiagram}) with $\Z^{n}$ and $\Z^{2n}$ replaced by $\Z^{3}\oplus \Z^{3}$ and $\Z^{4}\oplus \Z^{4}$, where the middle map $\Z^{4}\oplus \Z^{4}\to \Z^{3}\oplus \Z^{3}$ is defined by
$x_{i}\mapsto \bar{x}_{i}$, $y_{i}\mapsto \bar{y}_{i}$, $x_{4}\mapsto -\bar{x}_{1}-\bar{x}_{2}-\bar{x}_{3}$, and $y_{4}\mapsto -\bar{y}_{1}-\bar{y}_{2}-\bar{y}_{3}$ for $1\leq i\leq 3$. Then, it follows that we find the following $\Z$-basis of $T_{H}^{*}$:
\begin{equation}\label{thnewbasis}
\{\bar{x}_{1}-\bar{x}_{2}, \bar{x}_{1}-\bar{x}_{3}, \bar{y}_{1}-\bar{y}_{2}, \bar{y}_{1}-\bar{y}_{3}, 2\bar{x}_{1}+2\bar{y}_{1}, 2\bar{x}_{1}-2\bar{y}_{1}\}.
\end{equation}

As the group $S^{2}(\bar{T}_{H}^{*})^{W}$ with $W=(S_{4})^{2}$ is generated by $q_{1}=\sum_{i}\bar{x}_{i}^{2}+\sum_{i<j}\bar{x}_{i}\bar{x}_{j}$ and $q_{2}=\sum_{i}\bar{y}_{i}^{2}+\sum_{i<j}\bar{y}_{i}\bar{y}_{j}$ with $1\leq i\neq j\leq 3$, the same argument as in the proof of Proposition \ref{propsl} together with (\ref{thnewbasis}) shows that $S^{2}(T_{H}^{*})^{W}$ is generated by
\[\{4q_{1}+4q_{2}, 2q_{1}+6q_{2}\}.\]

Since the image of the $2$nd Chern class map is generated by $\{8q_{1}, 4q_{1}+4q_{2}\}$, the indecomposable subgroup of $H$ is isomorphic to $\Z/2\Z$ generated by the class of $2q_{1}+6q_{2}$.
\end{proof}

\begin{remark}
Indeed, in $S^{2}(T_{H}^{*})^{W}$ the element $2q_{1}+6q_{2}$ is written as \[2(\bar{x}_1-\bar{x}_2)\{(\bar{x}_1-\bar{x}_2)+(\bar{x}_1-\bar{x}_3)
-2(\bar{x}_1+\bar{y}_1)-2(\bar{x}_1-\bar{y}_1)\}+12(\bar{x}_1+\bar{y}_1)\{(\bar{x}_1+\bar{y}_1)-(\bar{x}_1 -\bar{y}_1)\}\]
\[+2(\bar{y}_1-\bar{y}_2)\{3(\bar{y}_1-\bar{y}_2)+3(\bar{y}_1-\bar{y}_3)
-6(\bar{x}_1+\bar{y}_1)+6(\bar{x}_1-\bar{y}_1)\}+12(\bar{x}_1-\bar{y}_1)\{(\bar{y}_1-\bar{y}_3)+(\bar{x}_1-\bar{y}_1)\}\]
\[+2(\bar{x}_1-\bar{x}_3)
\{(\bar{x}_1-\bar{x}_3)-2(\bar{x}_1+\bar{y}_1)-2(\bar{x}_1-\bar{y}_1)\}+6(\bar{y}_1-\bar{y}_3)\{(\bar{y}_1-\bar{y}_3)-2(\bar{x}_1+\bar{y}_1)\}.\]

\end{remark}

\begin{lemma}
Let $\bar{X}$ be the generic variety associated to $(\gSL_{4}\times \gSL_{4})/\gmu$, where $\gmu=\{(\lambda_{1}, \lambda_{2})\in (\gmu_{4})^{2}\,|\, \lambda_{1}^{2}\lambda_{2}^{2}=1\}$. Then, the group $\CH^{2}(\bar{X})_{\torsion}$ is trivial.
\end{lemma}
\begin{proof}
Let $X=\SB(A_{1})\times \SB(A_{2})$, where $A_{1}$ and $A_{2}$ are division algebras of degree $4$ over a field $K/F$ such that
\[\ind(A_{1}\tens A_{2})=\ind(A_{1}\tens A_{2}^{\tens 3})=\ind(A_{1}^{\tens 3}\tens A_{2})=16,\,\, \ind(A_{i}^{\tens 2})=2,\]
\[\ind(A_{1}\tens A_{2}^{\tens 2})=\ind(A_{1}^{\tens 2}\tens A_{2})=4,\,\, \ind(A_{1}^{2}\tens A_{2}^{2})=1\]
for $i=1, 2$. By (\ref{Quillenbasis}), we have the following basis of $K_{0}(X)$
\begin{equation*}\label{AoneAtwo}
\{1, 4x_{i}, 2x_{i}^{2}, 16x_{1}x_{2}, 4x_{i}^{2}, 4x_{1}^{2}x_{2}, 4x_{1}x_{2}^{2}, x_{1}^{2}x_{2}^{2}, 16x_{1}x_{2}^{3}, 16x_{1}^{3}x_{2}, 4x_{1}^{3}x_{2}^{2}, 4x_{1}^{2}x_{2}^{3}, 16x_{1}^{3}x_{2}^{3}\},
\end{equation*}
where $x_{i}$ is the pullback of the tautological line bundle on $\P^{3}$ over a splitting field $E$ of $X$.

Let $\delta_{n}=|\Gamma^{n/n+1}(X_{E})/\Im(\res^{n/n+1})|/|K^{n}_{0}(X_{E})/K^{n}_{0}(X)|$, where $1\leq n\leq 6$ and $K^{n}(X_{E})$ (respectively, $K^{n}_{0}(X)$) is the codimension $n$ part of $K_{0}(X_{E})$ (respectively, $K_{0}(X)$). We find upper bounds of $\delta_{n}$ using case by case analysis. We use another basis for $K_{0}(X)$ the above basis by replacing $x_{i}$ by $y_{i}:=x_{i}-1$. As $x_{1}^{2}x_{2}^{2}=(y_{1}+1)^{2}(y_{2}+1)^{2}\in \Gamma^{1}(X)$, we have $2(y_{1}+y_{2})\in \Im(\res^{1/2})$, thus $2\delta_{1}\leq 1$.

In codimension $2$, it follows from $2(y_{1}+1)^{2}(y_{2}+1)^{2}\in \Gamma^{2}(X)$ that $8y_{1}y_{2}\in \Im(\res^{2/3})$. Since $c_{2}(4x_{i})=6y_{i}^{2}\in \Gamma^{2}(X)$, we obtain $2\delta_{2}\leq 1$. In codimension $3$, we have $c_{3}(4x_{i})=4y_{i}^{3}\in \Gamma^{3}(X)$ and $2(y_{1}+y_{2})\cdot 2y_{i}^{2}\in \Im(\res^{3/4})$, thus $\delta_{3}\leq 1$. In codimension $4$, we see that $2y_{1}^{2}\cdot 2y_{2}^{2}\in \Gamma^{2}(X)\cdot \Gamma^{2}(X)\subseteq \Gamma^{4}(X)$. Furthermore, it follows by a calculation of $c_{4}(4x_{1}^{2}x_{2})\in \Gamma^{4}(X)$ and $c_{4}(4x_{1}x_{2}^{2})\in \Gamma^{4}(X)$ that we have $8y_{1}y_{2}^{3}$, $8y_{1}^{3}y_{2}\in \Im(\res^{4/5})$. Hence, $\delta_{4}\leq 1$.

In codimension $5$, we get $8y_{1}^{3}y_{2}^{2}=4y_{1}^{3}\cdot 2y_{2}^{2}\in \Gamma^{3}(X)\cdot \Gamma^{2}(X)\subseteq \Gamma^{5}(X)$. Similarly, $8y_{1}^{2}y_{2}^{3}\in \Gamma^{5}(X)$. Therefore, $\delta_{5}\leq 4$. Finally, as $4y_{1}^{3}\cdot 4y_{2}^{3}\in \Gamma^{3}(X)\cdot \Gamma^{3}(X)\subseteq \Gamma^{6}(X)$ we have $\delta_{6}\leq 1$. In conclusion, we obtain $|\!\oplus \Gamma^{n/n+1}(X)_{\torsion}|=1$. In particular, the group $\CH^{2}(X)$ is torsion-free.

Consider the generic variety $\bar{X}$ of $H:=(\gSL_{4}\times \gSL_{4})/\gmu$ over the function field $F(U/B)$, where $U$ is an open subset of a generically free representation of $H$. Since the canonical morphism $\bar{X}\to X$ is an iterated projective bundle, it follows by the projective bundle theorem that $\CH^{2}(\bar{X})_{\torsion}\simeq \CH^{2}(X)_{\torsion}$. Hence, the result follows.\end{proof}

This lemma and Lemma \ref{lemmaforslfour} together with (\ref{mainMNZ}) yield the following result:

\begin{proposition}\label{lastprop}
Let $H$ be a factor group of $\gSL_{4}\times \gSL_{4}$ by a central subgroup $\gmu=\{(\lambda_{1}, \lambda_{2})\in (\gmu_{4})^{2}\,|\, \lambda_{1}^{2}\lambda_{2}^{2}=1\}$. Then, the group $\Inv^{3}(H)_{\ind}$ is cyclic of order $2$ and any degree $3$ normalized invariant of $H$ is semi-decomposable.
\end{proposition}

\paragraph{\bf Acknowledgments.}
The author would like to express his deep gratitude to Alexander Merkurjev for introducing the problem and giving helpful suggestions to him. This work was partially supported by an internal fund from KAIST, TJ Park Junior Faculty Fellowship of POSCO TJ Park Foundation, and National Research Foundation of Korea (NRF) funded by the Ministry of Science, ICT and Future Planning (2013R1A1A1010171).


\begin{thebibliography}{10}


\bibitem{Baek}
S.~Baek, \emph{Codimension 2 cycles on products of projective homogeneous surfaces}, Preprint available at http://arxiv.org/abs/1307.0669.

\bibitem{BM09}
S. Baek and A. Merkurjev, \emph{Invariants of simple algebras}, Manuscripta Math. \textbf{129} (2009), no. 4, 409--421.

\bibitem{BR}
H.~Bermudez and A.~Ruozzi, \emph{Degree 3 cohomological invariants of groups that are neither simply connected nor adjoint}, J.~Ramanujan Math. Soc. \textbf{29} (2014), no.~4, 465--481.

\bibitem{BM}
S.~Blinstein and A.~Merkurjev, \emph{Cohomological invariants of algebraic tori}, Algebra Numer Theory \textbf{7} (2013), no.~7, 1643--1684.

\bibitem{Skip}
S.~Garibaldi, \emph{Cohomological invariants: Exceptional groups and spin groups}, Mem. Amer. Math. Soc. \textbf{200} (2009), no.~937.

\bibitem{GMS}
S.~Garibaldi, A.~Merkurjev, and J.-P. Serre, \emph{Cohomological invariants in Galois cohomology}, American Mathematical Society, Providence, RI, (2003).

\bibitem{MerDeg}
A.~Merkurjev, \emph{Degree three cohomological invariants of semisimple groups}, To appear in JEMS.

\bibitem{MerUnramified}
A.~Merkurjev, \emph{Unramified degree three invariants of reductive groups}, Preprint available at http://www.mathematik.uni-bielefeld.de/LAG/man/543.html.

\bibitem{MNZ}
A.~Merkurjev, A.~Neshitov, K.~Zainoulline, \emph{Cohomological Invariants in degree 3 and torsion in the Chow group of a versal flag}, To appear in Compos. Math.

\bibitem{IzhKar}
O.~T.~Izhboldin and N.~A.~Karpenko, \emph{Generic splitting fields of central simple algebras: {G}alois cohomology and nonexcellence}, Algebr. Represent. Theory~\textbf{2} (1999), no. 1, 19--59.

\bibitem{Kar95}
N.~Karpenko, \emph{On topological filtration for {S}everi-{B}rauer varieties}, Proc. Sympos. Pure Math. \textbf{58} (1995), 275--277.

\bibitem{Kar}
N.~Karpenko, \emph{Codimension $2$ cycles on Severi-Brauer varieties}, K-Theory \textbf{13} (1998), no. 4,
305--330.

\bibitem{KLST}
M.-A.~Knus, T. Y.~Lam, D. B.~Shapiro and J.-P.~Tignol, \emph{Discriminants of involutions on biquaternion algebras}, In: {$K$}-theory and algebraic geometry: connections with quadratic forms and division algebras ({S}anta {B}arbara, {CA}, 1992), Proc. Sympos. Pure Math. \textbf{58}, 279--303, Amer. Math. Soc., Providence, RI (1995).

\bibitem{Siv}
A.~S.~Sivatski, \emph{The chain lemma for biquaternion algebras}, J. Algebra \textbf{350} (2012), 170--173.


\bibitem{Qui}
D.~Quillen, \emph{Higher algebraic {$K$}-theory. {I}}, Lecture Notes in Math. \textbf{341}, Springer (1973), 85--147.

\end{thebibliography}
\end{document}